\documentclass[10pt, reqno]{amsart}
\usepackage{cite}

\makeatletter
\def\LaTeX{\leavevmode L\raise.42ex
    \hbox{\kern-.3em\size{\sfsize}{0pt}\selectfont A}\kern-.15em\TeX}
\makeatother

\newcommand{\BibTeX}{{\rm B\kern-.05em{\sc
          i\kern-.025emb}\kern-.08em\TeX}}

\makeatletter
\label{e:dispaa}
\label{e:dispau}
\label{e:dispav}
\label{e:dispaw}
\label{e:dispax}
\makeatother

\usepackage[english]{babel}
\usepackage{amsmath}
\usepackage{mathrsfs}    
\usepackage{amssymb}
\usepackage{amsthm}
\usepackage{graphicx,color}
\usepackage{hyperref}
\usepackage{verbatim} 

\definecolor{verde}{rgb}{0,0.35,0.1} 
\definecolor{rosso}{rgb}{0.7,0,0}
\definecolor{blue}{rgb}{0,0,1}
\definecolor{viola}{rgb}{0.6,0,0.4}

\renewcommand{\phi}{\varphi}

\usepackage[applemac]{inputenc}

\newtheorem{thm}{Theorem}[section]
\newtheorem{prop}[thm]{Proposition}

\newtheorem{lemma}[thm]{Lemma}

\newtheorem*{key}{Keywords}

\newtheoremstyle{definition2}{\topsep}{\topsep}%
     {}
     {}
     {\bfseries}
     {.}
     {.5em}
     {\thmnumber{(#2)}\thmname{ #1}\thmnote{ #3}}

\theoremstyle{definition}
\newtheorem{rem}[thm]{Remark}

\def\ep{\varepsilon}
\def\vphi{\varphi}

\def\cC{\mathcal{C}}

\def\cH{\mathcal{H}}

\def\cM{\mathcal{M}}
\def\cN{\mathcal{N}}

\def\cS{\mathcal{S}}

\def\N{\mathbb{N}}

\def\R{\mathbb{R}}

\def\hbar{\bar{h}}

\def\dist{{\rm dist}}

\def\leq{\leqslant}
\def\geq{\geqslant}

\newcommand{\rst}[1]{\ensuremath{{\mathbin |}%
\raise-.5ex\hbox{$#1$}}}

\title{\sc Existence and symmetry results for competing variational systems}
\author{H. Tavares and T. Weth}

\date{\today}

\begin{document}
\maketitle

\begin{abstract}
In this paper we consider a class of gradient systems of type
$$
-c_i \Delta u_i + V_i(x)u_i=P_{u_i}(u),\quad u_1,\ldots, u_k>0 \text{ in }\Omega, \qquad u_1=\ldots=u_k=0 \text{ on } \partial \Omega,
$$ 
in a bounded domain $\Omega\subseteq \R^N$. Under suitable assumptions on $V_i$ and $P$, we prove the existence of ground-state solutions for this problem. Moreover, for $k=2$, assuming that the domain $\Omega$ and the potentials $V_i$ are radially symmetric, we prove that the ground state solutions are foliated Schwarz symmetric with respect to antipodal points. We provide several examples for our abstract framework.
\end{abstract}

\begin{key}
Competitive systems, elliptic gradient systems, foliated schwarz symmetry, ground states solutions, positive solutions.
\end{key}


\section{Introduction}
\label{sec:introduction}

Let $\Omega$ be a bounded domain in $\R^N$, $N\geq 1$ and $P:\R^k \to \R$ a
$\cC^2$-function for some positive integer $k$. Moreover, let $V_i \in
L^\infty(\Omega)$ for $i=1,\dots,k$, and let $c_1,\dots,c_k$ denote
positive constants. In this paper, we will
be concerned with the Dirichlet problem 
\begin{equation}
  \label{eq:1}
\left \{ 
  \begin{aligned}
  -&c_i \Delta u_i +V_i(x)u_i= P_{u_i}(u),\qquad\;u_1,\dots,u_k > 0 &&\qquad 
\text{in $\Omega$},\\  
&u_1=u_2=\dots=u_k=0 &&\qquad \text{on $\partial \Omega$}.
  \end{aligned}
\right.  
\end{equation}
where $P_{u_i}$ stands for $\frac{\partial
  P}{\partial u_i}$. Note that the elliptic system in (\ref{eq:1}) is
of gradient type. Under suitable assumptions on $P$, we will prove the
existence of ground state solutions of (\ref{eq:1}) which can be found
by minimizing an associated functional over a natural constraint. In certain cases, we will also provide a minimax characterization
of the solutions, and in the case $k=2$ we will deduce symmetry
properties of the solutions from this characterization in the case
where $\Omega$ is a radially symmetric domain in $\R^N$ and the potentials $V_i$ are also radially symmetric. We point out
that we are only interested in nontrivial solutions of (\ref{eq:1}) in
the sense that $u_i \not \equiv 0$ for $i=1,\dots,k$. Consider the
Hilbert space $\cH:=H^1_0(\Omega; \R^k)$ and the {\em Nehari type set}
\begin{equation}\label{eq:Nehari}
\cN:=\Bigl\{ u\in \cH: u_i \ge 0, u_i\not\equiv 0 \text{ and } 
\!\!\int_\Omega \!
\bigl(c_i |\nabla u_i|^2+V_i(x)u_i^2\bigr)\,dx=\!\int_\Omega \! \!P_{u_i}(u)u_i\, dx \quad \text{for $i=1,\dots,k$.} \}
\end{equation}
If $\Omega$ is of class $\cC^1$ and $u \in \cC^2(\overline
\Omega,\R^k)$ is a classical solution of (\ref{eq:1}) with nontrivial components, then
we may multiply the $i$-th equation in (\ref{eq:1}) with $u_i$ and integrate by parts
to see that $u$ belongs to $\cN$. Therefore $\cN$ is a natural
constraint for solutions of (\ref{eq:1}). If moreover $P$ satisfies
suitable growth assumptions (see assumption (P1) below), then weak solutions
are precisely the critical points of the energy functional $E: \cH\to \R$ defined by
\begin{equation*}
\label{eq:energy}
E(u)=\frac{1}{2}\sum_{i=1}^k \int_\Omega
\bigl(c_i |\nabla u_i|^2+V_i(x)u_i^2\bigr)\,dx-\int_\Omega P(u(x))\, dx,
\end{equation*}  
A natural but not straighforward approach to find solutions of (\ref{eq:1}) is to minimize
$E$ on $\cN$. This approach has been carried out successfully in the
scalar case $k=1$ (see e.g. the recent survey \cite{szulkin-weth} and
the references therein) and also for special classes of
elliptic systems, see e.g. \cite{CTV1,DWW,linwei1,linwei2}. In this paper we will
consider a general class of functions $P$ for which we can show that 
minimizers of $E$ on $\cN$ exist and are indeed
solutions of~(\ref{eq:1}). Such solutions then also minimize the
energy $E$ among the set of solutions and therefore will be called
{\em ground state solutions}. It is natural to expect that in the case
where the underlying domain $\Omega$ and the potentials $V_i$ are radially symmetric, these
ground state solutions inherit at
least partially the symmetry of $\Omega$ and $V_i$. In a general framework, a
principle of symmetry inheritance of constrained minimizers of
integral functionals was recently proved by Mari{\c{s}}
\cite{maris:09}. In particular, the following statement can be deduced
from \cite[Theorem 1]{maris:09}:\\
{\em If $k \le
N-2$, $\Omega$ is radially symmetric and every minimizer of $E$ on $\cN$ is a solution of (\ref{eq:1}) (and
therefore a $\cC^1$-function on
$\Omega$), then every minimizer of $E$ on $\cN$
is radially symmetric with respect to a $k$-dimensional subspace $W$
of $\R^N$, i.e., $u(x)=u(y)$ for every $x,y \in \Omega$ such that $x-y
\in W^\perp$ and $\dist(x,W)=\dist(y,W)$.}\\
We stress that this symmetry result does not
depend on further assumptions on $P$. Much more is known in the
special case where the underlying domain $\Omega$ is a ball, $V_i\equiv 0$, and the system (\ref{eq:1}) is
{\em cooperative}, i.e. $P_{u_i u_j}= \frac{\partial^2
  P}{\partial u_i u_j} \ge 0$ for all $u \in \R^k$ and
$i,j=1,\dots,k$. In this case, every 
solution of (\ref{eq:1}) is in fact radially symmetric (with respect
to $W= \{0\}$) and decreasing
in the radial variable by the general symmetry result of Troy
\cite{troy} for cooperative systems. We note that Troy's result is
proved via the moving plane method and therefore relies strongly on
the cooperativity assumption. In the present paper, we are interested in
the complementary case of non-cooperative competition-type systems
which have been at the center of growing attention in recent years,
see e.g. \cite{CTV1,CTV2,DWW,CLLL,nttv2, norisramos, bartschdancerwang, tianwang} and references therein. In this case ground state
solutions are nonradial in general even if the underlying data is
radially symmetric (see Remark~\ref{sec:some-special-system} below). Nevertheless, we shall see below that, at least in
the two-component case, the competitive character of the system also
leads to an improvement of Mari{\c{s}}' symmetry result mentioned above.

In order to state or main results, we define the cone  
$$
C^+:= \{u=(u_1,\ldots, u_k)\in \R^k:\ u_i\geq 0 \text{ for all $i$}\}
$$
and we impose the following assumptions on the functions $V_i$,
$i=1\dots,k$ and $P$:
\begin{itemize}
\item[(P0)] $V_i \in L^\infty(\Omega)$ and $\inf_{\Omega} V
  >- c_i \lambda_1(\Omega)$ for $i=1,\dots,k$, where $\lambda_1$ is the
  first Dirichlet eigenvalue of the Laplacian on $\Omega$; 
\item[(P1)] there exists $2<p<2^\ast$ such that
$$
| P_{u_i u_j}(u) | \leq C(1+\sum_{i=1}^k |u_i|^{p-2})\qquad \text{ for
  every $u\in C^+$, $i,j=1,\ldots, k$,}
$$
where $2^\ast=2N/(N-2)$ if $N\geq 3$, $2^\ast=+\infty$ otherwise.
\item[(P2)] $P(0,\ldots,0)=0$ and $P_{u_i}(u_1,\ldots,u_{i-1},0,
  u_{i+1},\ldots, u_k)=0$ for every $u\in C^+$ and $i=1,\dots,k$.
\item[(P3)] $P_{u_i}(u)u_i  \leq P_{u_i}(0,\ldots, u_i,\ldots,
  0)u_i\neq 0$ for every $u \in C^+$ and every $i \in \{1,\dots,k\}$
  with $u_i \neq 0$;
\item[(P4)] there exists $\alpha>0$ such that the matrix
$$
M(u):=\Bigl( \delta_{ij} (1+\alpha) P_{u_i}(u)u_i-P_{u_i u_j}(u) u_i u_j\Bigr)_{i,j=1,\ldots, k}
$$
$$
=\left(
\begin{array}{cccc}
(1+\alpha) P_{u_1}(u)u_1-P_{u_1 u_1}(u)u_1^2 &\ldots & -P_{u_1 u_k}(u) u_1 u_k \\
-P_{u_1 u_2}(u) u_1 u_2 					   &  \ldots & -P_{u_2 u_k}(u) u_2 u_k \\
\vdots 								  &  \ddots   	& \vdots \\
-P_{u_1 u_k} (u) u_1 u_k					& \ldots  &  (1+\alpha) P_{u_k}(u)u_k-P_{u_k u_k}(u)u_k^2
\end{array}
\right)
$$
is negative semidefinite for $u\in C^+$.
\end{itemize}
Note that condition (P3) is a weak competitivity assumption for the
system~(\ref{eq:1}). Condition (P4) can be seen has a
generalization of an Ambrosetti-Prodi condition, and it has appeared
before in the papers \cite{CTV1, CTV2} (actually, we will see in the
last section that our assumptions are more general than the ones
considered in the mentioned papers). We now put 
\begin{equation}
  \label{eq:19}
c=\inf_{u\in \cN} E(u),
\end{equation}
where $\cN$ was defined in \eqref{eq:Nehari}.
Our first main result shows that minimizers of $E$ on $\cN$ exist and are ground
state solutions of (\ref{eq:1}). More precisely, we have:
\begin{thm}\label{thm:existence}
Suppose that (P0)--(P4) holds. Then there exists $u=(u_1,\ldots,u_k)\in \cH$, with $u_i>0$ for all $i$, such that 
\begin{equation*}\label{eq:c_is_critical_level}
E'(u)=0 \qquad \text{and} \qquad E(u)=c.
\end{equation*}
Moreover, every minimizer of $E|_{\cN}$ is a solution of (\ref{eq:1}).
\end{thm}

It is worth discussing the scalar case $k=1$ in some detail. In this
case, the assumptions $(P1)$--$(P4)$ above reduce to requiring $P \in
\cC^2(\R)$, $P(0)=P'(0)=0$ as well as
$|P''(u)| \leq C(1+|u|^{p-2})$ and $0< (1+\alpha)P'(u)
\le P''(u) u$ for every $u>0$ with $p$ as in $(P1)$ and some $\alpha>0$. From these assumptions, it
follows that $0$ is a local minimum for the corresponding functional
$E:H^1_0(\Omega) \to \R$. Moreover, for any $u \in H^1_0(\Omega)
\setminus \{0\}$ the
function $\phi_u: [0,\infty) \to \R$, $\phi_u(t)= E(tu)$ satisfies
$\phi_u(0)=0$ and $\phi_u(t) \to -\infty$ as $t \to \infty$, and 
$\phi_u$ has a unique maximum $t_u$ such that $t_u u \in \cN$. As a
consequence, in the scalar case $k=1$ we have the minimax characterization
\begin{equation}
  \label{eq:2}
c= \inf_{u \in H^1_0(\Omega) \setminus \{0\}} \sup_{t \ge 0}E(tu)
\end{equation}
In the case $k>1$ assumptions $(P1)$--$(P4)$ do not impose such a
simple mountain pass geometry for the functional $E$. Nevertheless,
for a large class of functions $P$ satisfying these assumptions one
may generalize the minimax
characterization (\ref{eq:2}). For this we consider the set 
\begin{equation*}
\cM=\{u\in \cH:\ u_i \ge 0, u_i \not\equiv 0\; \text{for $i=1,\dots,k$ and }\ E(t_1 u_1,\ldots, t_k u_k)\to -\infty \text{ as } |t_1|+\ldots +|t_k|\to +\infty\}.
\end{equation*}
We then have: 

\begin{thm} 
\label{sec:introduction-1}
Suppose that $(P0)$--$(P4)$ and the following condition holds.
\begin{itemize}
\item[(P5)] $P_{u_i u_i}(u_1,\ldots, u_{i-1},0,u_{i+1},\ldots,
  u_k)\leq 0$ for $u \in C^+$ and $i=1,\dots,k$.
\end{itemize}
Then we have
\begin{equation}
  \label{eq:3}
c\leq \inf_{u\in \cM} \sup_{t_1,\ldots, t_k \ge 0} E(t_1 u_1,\ldots, t_k u_k).
\end{equation}
If moreover there exists $u\in \cN$ with $E(u)=c$ and such that $u \in
\cM$, then equality holds in (\ref{eq:3}).
\end{thm}
The crucial step in the proof of Theorem~\ref{sec:introduction-1} is to prove
that, for fixed $u \in \cM$, the function 
$$
(t_1,\ldots, t_k) \mapsto E(t_1 u_1,\ldots, t_k u_k)
$$
has precisely one critical point in $C^+ \setminus \{0\}$ which is a
global maximum of this function in $C^+ \setminus \{0\}$. This fact
will also be used in the proof of our main symmetry result
Theorem~\ref{coro:main} below. While assumptions $(P1)$--$(P5)$ do not guarantee that every function $u \in
\cH$ with $u_i\not\equiv 0$ for $i=1,\dots,k$ is contained in $\cM$, 
below we will present classes of functions $P$ ensuring that
$\cN \subset \cM$, so that equality holds in (\ref{eq:3}). One
explicit example we consider is the class of functions 
\begin{equation}
  \label{eq:5}
P\in \cC^2(\R^k), \qquad P(u_1,\dots,u_k)= 
\sum_{i=1}^k \frac{\lambda_i}{p}|u_i|^p - \sum_{\stackrel{i,j=1}{i
    \not= j}}^n \beta_{ij}|u_i|^{q_i} |u_j|^{q_j}
\end{equation}
which leads to the system
\begin{equation}\label{eq:application}
\left\{
\begin{array}{l}
-c_i \Delta u_i+V_i(x) u_i=\lambda_i u_i^{p-1}-q_i u_i^{q_i-1}\sum \limits_{j\neq i} \beta_{ij} u_j^{q_j},\\
u_i\in H^1_0(\Omega), \ u_i> 0 \text{ in }\Omega,
\end{array}
\right.
\end{equation}
Here we assume $2<p<2^\ast$ and 
\begin{equation}
  \label{eq:4}
\lambda_i>0,\quad \beta_{ij}=\beta_{ji}\geq 0,\quad q_i \ge 2
\quad \text{and}\quad p\geq
q_i+q_j \qquad \text{for $i,j=1,\dots,k$, $j\neq i$.} 
\end{equation}
A system of this kind also appears in \cite{quitnersouplet}. We point out that one may divide out (or replace by arbitrary positive
constants) the factors $q_i$ in front of the sums in
(\ref{eq:application}) without changing the nature of the system
simply by adjusting the values of $c_i$, $V_i(x)$ and
$\lambda_i$. Therefore, the cubic system
\begin{equation*}
  \label{eq:14}
-\Delta u_i + V_i(x)u_i=\lambda_iu_i^3- u_i \sum_{j\neq i} \beta_{ij} u_j^2, \qquad i=1,\ldots,k,
\end{equation*}
arising in the theory of Bose-Einstein condensation and in nonlinear
optics (see e.g. \cite{BEC, sirakov}) can be seen as a special case of (\ref{eq:application}).

In Section~\ref{sec:applications} below we will show that the class of
functions $P$ given by (\ref{eq:5}) satisfies (P1)--(P5), whereas we
also have $\cN \subset \cM$ so that equality holds in (\ref{eq:3}).

Our final main result is concerned with symmetry properties of ground
state solutions of (\ref{eq:1}) in the case where the underlying
domain $\Omega \subset \R^N$ and the potentials $V_i$ are radial. For this we recall the notion
of foliated Schwarz symmetry. A function $u:\Omega \to \R$ is called
foliated Schwarz  symmetric with respect to some unit vector $p \in
\R^N$ if for a.e. $r>0$ such that $\partial B_r(0)\subset \Omega$ and
for every $c\in \R$ the restricted superlevel set $\{x\in \partial
B_r(0):\ u(x)\geq c\}$ is either equal to $\partial B_r(0)$ or to a
geodesic ball in $\partial B_r(0)$ centered at $rp$. In other words,
$u$ is foliated Schwarz symmetric if $u$ is axially symmetric 
with respect to the axis $\R p$ (i.e. radially symmetric with respect
to the subspace spanned by $p$ in the sense defined above) and nonincreasing in the polar angle
$\theta=\arccos(\frac{x}{|x|}\cdot p) \in [0,\pi]$. We have to
restrict our attention to the case of two components, and we will
write $(u,v)$ in place of $(u_1,u_2)$ in the following. Hence we
consider the system
\begin{equation}
\label{eq:system2eq-0}
\left \{ 
  \begin{aligned}
&-c_1 \Delta u+V_1(x)u=P_u(u,v),\quad -c_2 \Delta v+V_2 (x)u=P_v(u,v)\qquad \text{in $\Omega$},\\  
&u,v > 0 \quad 
\text{in $\Omega$},\qquad u=v=0 \quad \text{on $\partial \Omega$}.
  \end{aligned}
\right.  
\end{equation}

\begin{thm}\label{coro:main}
Suppose that $\Omega \subset \R^N$ is a radial domain,  that $V_1,V_2$
are radial functions, and suppose that (P0)--(P5) hold for $k=2$. Suppose moreover that
\begin{itemize}
\item[(P6)] $P_{uv}(s,t)<0$ for every $s,t>0$.
\end{itemize}
Let $(u,v)\in C^2(\Omega,\R^2)\cap C(\overline\Omega,\R^2)$ be a classical solution of \eqref{eq:system2eq-0} minimizing $E|_\cN$. If $(u,v)\in \cM$, then $u$ and $v$ are foliated Schwarz symmetric with respect to antipodal points.
\end{thm}

We note that -- for a class of systems with competition --
Theorem~\ref{coro:main} improves the symmetry result of Mari{\c{s}} \cite{maris:09} 
which in contrast to Theorem~\ref{coro:main} only yields radial symmetry with
respect to a two-dimensional subspace. For a larger number $k \ge 3$ of
components, it remains open whether the symmetry result of Mari{\c{s}} can
be improved as well, although foliated Schwarz symmetry should not
be expected. We also note that assumption (P6) implies (P3) for $k=2$,
so we could have neglegted assumption (P3) in Theorem~\ref{coro:main}.

In the following theorem, we summarize our results for the special class of systems
(\ref{eq:application}).

\begin{thm}
\label{coro:main-1}
Let $P$ be given by (\ref{eq:5}) and suppose that (\ref{eq:4})
holds. Then $\cN \subset \cM$, and  
\begin{equation}
  \label{eq:3.1}
\inf_\cN E= \inf_{u\in \cM} \sup_{t_1,\ldots, t_k \ge 0} E(t_1 u_1,\ldots, t_k u_k)
\end{equation}
is attained. Moreover, every minimizer $u \in \cN$ of $E|_{\cN}$ 
is a classical solution $u 
\in C^2(\Omega,\R^k)\cap C(\overline\Omega,\R^k)$ of
(\ref{eq:application}). Moreover, if $k=2$, $\Omega$ is a radial domain
and $V_1,V_2$ are radial functions, then
every ground state solution is such that $u$ and $v$ are foliated Schwarz symmetric with respect to antipodal points.    
\end{thm}

The paper is organized as follows. In
Section~\ref{sec:some-general-notions} we will collect some
preliminary results and give the proof of
Theorem~\ref{thm:existence}. In Section~\ref{sec:an-altern-char} we
will prove Theorem~\ref{sec:introduction-1} and therefore give --
under additional assumptions --
a minimax characterization of the value $c$ defined in (\ref{eq:19}).
The key result in this section is
Proposition~\ref{prop:unique_maximum} below which will also be used in
Section~\ref{sec:symmetry results} where we prove our symmetry results
for the special class of two-component systems (\ref{eq:system2eq-0}).
In particular, the proof of Theorem~\ref{coro:main} is contained in
Section~\ref{sec:symmetry results}. In Section~\ref{sec:applications}
we consider special classes of systems satisfying our general
assumptions. In particular, we will consider system
(\ref{eq:application}) and prove Theorem~\ref{coro:main-1} in this
section. We close this section with an application of our symmetry
results to a different setting.


\section{Some preliminaries and the existence of ground state solutions}
\label{sec:some-general-notions}

We will assume conditions (P0)--(P4)
from now on, and we start with some general remarks on
problem~(\ref{eq:1}). First, by the transformation
\begin{equation}
  \label{eq:12}
\cH \to \cH, \qquad u \mapsto 
\tilde u=(\sqrt{c_1}u_1,\dots,\sqrt{c_k}u_k)
\end{equation}
problem (\ref{eq:1}) is reduced to the special case $c_1,\dots,c_k=1$
with $V_i$ replaced by $\frac{V_i}{c_i}$ and $P$ replaced by 
$$
\tilde P \in \cC(\R^k),\qquad \tilde P(v_1,\dots,v_k)=
P\left(\frac{v_1}{\sqrt{c_1}},\dots,\frac{v_k}{\sqrt{c_k}}\right).
$$
Moreover, the transformation (\ref{eq:12}) maps the 
corresponding Nehari sets and the sets $\cM$ into each other and preserves the value of
the corresponding energy functionals. Hence we may assume from now on
that 
\begin{equation*}
  \label{eq:13}
c_1= \dots = c_k=1.  
\end{equation*}
As we will be interested in nonnegative solutions only, we will 
also assume from now on that 
\begin{equation}
  \label{eq:8}
P(u_1,\ldots, u_i,\ldots,
u_k)=P(u_1,\ldots,|u_i|, \ldots,u_k)\qquad \text{for every $u \in \cH$
  and $i=1,\dots,k$.}  
\end{equation}
Observe that this is consistent with the fact that $P\in C^2(\R^k)$ by the second assumption in (P2)). We then deduce from (P1)
that $E$ is a $C^1$-functional on $\cH$, and that 
\begin{equation}
  \label{eq:6}
E(u_1,\dots,u_k)=E(|u_1|,\dots,|u_k|) \qquad \text{for every $u \in \cH$.}  
\end{equation}

In the following, we will write
$$
\|u\|_i^2:=\int_\Omega \bigl(|\nabla u(x)|^2+V_i(x)u^2\bigr)\,dx
\qquad \text{for $u \in H^1_0(\Omega)$}
$$
and $i=1,\dots,k$, and we note that the norms $\|\cdot\|_i$ are
equivalent to the standard $H^1_0$-norm as a consequence of assumption (P0).
We also denote the $L^p$-norm by $\|\cdot\|_{L^p}$. We then define the Nehari manifold
\begin{equation}
  \label{eq:7}
\cN_*:=\Bigl\{ u\in \cH: u_i\not\equiv 0 \text{ and } \|u_i\|_i^2=\int_\Omega P_{u_i}(u)u_i\, dx \quad \text{for $i=1,\dots,k$.} \}
\end{equation}
and we note that for $u=(u_1,\dots,u_k) \in \cH$ we have the equivalence
\begin{equation*}
  \label{eq:9}
u \in \cN_* \qquad \Longleftrightarrow \qquad (|u_1|,\dots,|u_k|)  \in
\cN   
\end{equation*}
with $\cN$ defined in (\ref{eq:Nehari}). Combining this with
(\ref{eq:6}), we deduce that 
\begin{equation}
  \label{eq:10}
\text{$\inf_\cN E= \inf_{\cN_*}E$, so that every minimizer $u \in
  \cN$ of $E|_\cN$ also minimizes $E|_{\cN_*}$.} 
\end{equation}

We now collect some easy consequences of assumptions (P1)--(P4).

\begin{lemma}\label{lemma:properties_of_P} 
\begin{itemize}
\item[(i)] There exists $C>0$ such that 
\begin{equation}\label{eq:upper_bounds_P}
|P(u)| \leq C(1+\sum_{i=1}^k |u_i|^p) \quad \text{and}\quad
|P_{u_i}(u)u_i| \leq C(1+\sum_{i=1}^k |u_i|^p)\qquad \text{for all $u\in C^+$.}
\end{equation}
\item[(ii)]  For every $u\in C^+$ we have $(2+\alpha)P(u) \leq \sum \limits_{i=1}^k P_{u_i}(u) u_i$. 
\item[(iii)] $P_{u_i}(t e_i)/t\to +\infty$ as $t\to+\infty$ for
  $i=1,\dots,k$, where $e_i$ denotes the $i$--th coordinate vector.
\end{itemize}
\end{lemma}

\begin{proof}
(i) To see this, first take $\vphi_1(t)=P_{u_i}(tu)u_i$. We have, for $t\in [0,1]$, 
$$|\vphi_1'(t)|=|\sum_{j=1}^k P_{u_i u_j}(tu)u_iu_j|\leq \sum_{j=1}^k |P_{u_i u_j}(tu)u_iu_j|\leq \sum_{j=1}^k C |u_i|\,|u_j|\,(1+\sum_{i=1}^k |t u_i|^{p-2})\leq C' (1+ \sum_{i=1}^k |u_i|^p)  $$
which, together with the fact that $\vphi_1(0)=0$, implies the upper bound for $|P_{u_i}(u)u_i|$. The same reasoning implies the other condition.

(ii) Consider the function $\vphi_2(t)=(2+\alpha)P(tu)-\sum_{i=1}^k P_{u_i}(tu)tu_i$. For $t>0$, 
\begin{eqnarray*}
\vphi_2'(t) &=& (2+\alpha) \sum_{i=1}^k P_{u_i}(tu)u_i -\sum_{i,j=1}^k P_{u_i u_j}(tu) tu_i u_j-\sum_{i=1}^k P_{u_i}(tu)u_i\\
	     &=& (1+\alpha) \sum_{i=1}^k P_{u_i}(tu)u_i -\sum_{i,j=1}^k P_{u_i u_j}(tu) tu_i u_j \\
	     &=& \frac{1}{t} (1,\ldots, 1)\cdot M(tu) \cdot (1,\ldots,1)^T\leq 0.
\end{eqnarray*}
Then $\vphi_2(1)\leq \vphi_2(0)=0.$

(iii) From (P4) one can see that $(1+\alpha)P_{u_i}(te_i)\leq P_{u_i u_i}(t e_i)te_i$, which implies (as $P_{u_i}(t e_i)\neq 0$) that $P_{u_i}(t e_i)\geq C t^{1+\alpha}$ for some $C>0$ and for every $t>1$.
\end{proof}

The remainder of this section will be devoted to the proof of
Theorem~\ref{thm:existence}. For this we need to study the sets $\cN$ and $\cN_*$.

\begin{lemma} 
The Nehari manifold $\cN \subset \cH$ defined in (\ref{eq:Nehari}) is nonempty.
\end{lemma}
\begin{proof}
Take $k$ functions $w_1,\ldots, w_k\in H^1_0(\Omega)$ such that
$$
w_i\not\equiv 0, \quad w_i\geq 0 \text{ in } \Omega\ \text{for $i=1,\dots,k$}, \qquad \text{ and } \qquad w_i\cdot w_j\equiv 0 \text{ in } \Omega \quad \text{ whenever } i\neq j.
$$

Consider the functions $\vphi_i:\R^k\mapsto \R$ defined by
\begin{eqnarray*}
\vphi_i(t_1,\ldots,t_k)&=& \| t_i w_i\|_i^2-\int_\Omega P_{u_i} (t_1w_1,\ldots,t_i w_i,\ldots,t_k w_k)t_i w_i \, dx \\
			      &=& \| t_i w_i\|_i^2-\int_{\{w_i>0\}} P_{u_i} (0,\ldots,t_i w_i,\ldots,0)t_iw_i \, dx .
\end{eqnarray*}
Observe first of all that \eqref{eq:upper_bounds_P} and (P2) imply the existence of a constant $C>0$ such that
$$
\int_\Omega P_{u_i}(0,\ldots, t_iw_i,\ldots,0)t_iw_i\leq \frac{\|t_i w_i\|_i^2}{2}+C \| t_iw_i\|_i^p,
$$
and hence
$$
\vphi_i(t_1,\ldots, t_k)\geq t_i^2 \Bigl(\frac{\|w_i\|_i^2}{2}-Ct_i^{p-2}\|w_i\|_i^{p}\Bigr)>0
$$
for $t_i>0$ very close to zero, uniformly in $t_1,\ldots, t_{i-1}, t_{i+1},\ldots, t_k$. On the other hand, from Lemma \ref{lemma:properties_of_P}-(iii) we see that
$$
\vphi_i(t_1,\ldots,t_k)=  t_i^2 \Bigl(\|w_i\|_i^2-\int_{\{w_i>0\}} \frac{P_{u_i}(0,\ldots,t_i w_i,\ldots,0)}{t_iw_i}w_i^2\, dx\Bigr)\to -\infty
$$
as $t_i\to +\infty$, uniformly in $t_1,\ldots, t_{i-1}, t_{i+1},\ldots, t_k$. Hence we deduce the existence of $t_1,\ldots, t_k>0$ such that $\vphi_i(t_1,\ldots,t_k)=0\ \forall i$, and $(t_1 w_1,\ldots, t_k w_k)\in \cN$, which is non empty.
\end{proof}

\begin{lemma}\label{lemma:far_away_from_zero}
There exists $\gamma>0$ such that
$$
\|u_i\|_{L^p}, \|u_i\|_i \geq \gamma>0 \qquad \text{ for every } u\in \cN_*.
$$
\end{lemma}
\begin{proof}
By using (P2), (P3) and \eqref{eq:upper_bounds_P}, we know that for every $u\in \cN_*$ we have
\begin{eqnarray*}
\|u_i\|_i^2 &=& \int_\Omega P_{u_i}(u)u_i\, dx \leq \int_\Omega P_{u_i}(0,\ldots, u_i,\ldots, 0)u_i\, dx\\
		&\leq&\frac{1}{2}\|u_i\|_i^2+C_1\|u_i\|_{L^p}^p \leq \frac{1}{2}\|u_i\|_i^2+C_2\|u_i\|_i^p.
\end{eqnarray*}
Thus 
$$
\frac{1}{(2C_2)^{1/(p-2)}}\leq \|u_i\|_i, \quad \text{ and }\quad \frac{1}{(2C_2)^{2/(p-2)}}\leq C_1 \|u_i\|_{L^p}^p.
$$
\end{proof}

\begin{lemma}\label{lemma:N_is_manifold}
The set $\cN_*$ is a submanifold of $\cH$ of codimension $k$. Moreover, if $u\in \cN_*$ is such that $E|_{\cN_*}'(u)=0$, then $E'(u)=0$.
\end{lemma}

\begin{proof} The elements in $\cN_*$ are zeros of the  functional $F:\cH\to \R^k$, $u=(u_1,\ldots,u_k)\mapsto (F_1(u),\ldots,F_k(u))$ where, for each $i=1,\ldots,k$, $F_i$ is the $C^1(\cH,\R)$--functional defined by 
$$F_i(u)=\|u_i\|_i^2 - \int_\Omega P_{u_i}(u)u_i\, dx.$$ 
Denote by ${\bf T_u}$ the $k\times k$ matrix whose $i$-th line is the vector 
$$F'(u)(0,\ldots,u_i,\ldots,0)=(\partial_{u_i} F_1(u)u_i,\ldots, \partial_{u_i}F_k(u)u_i).$$
Given $u\in \cN_*$, for each $i$ we have
\begin{eqnarray*}
\partial_{u_i}F_i(u)u_i &=& 2\|u_i\|_i^2-\int_\Omega ( P_{u_i u_i}(u)u_i^2+P_{u_i}(u)u_i)\, dx\\
				   &=& -\alpha \|u_i\|_i^2 + (2+\alpha) \|u_i\|_i^2-\int_\Omega ( P_{u_i u_i}(u)u_i^2+P_{u_i}(u)u_i)\, dx\\
				   &=& -\alpha \|u_i\|_i^2+\int_\Omega ((1+\alpha)P_{u_i}(u)u_i-P_{u_i u_i}(u) u_i^2)\, dx,
\end{eqnarray*}
while for $j\neq i$
$$
\partial_{u_j}F_i(u)u_j=  - \int_\Omega P_{u_i u_j}(u) u_i u_j\, dx. 
$$ Thus,
$$
\mathbf{T_u} = \Bigl( -\alpha \delta_{ij} \|u_i\|_i^2 \Bigr)_{i,j}+ \Bigl( \int_\Omega (\delta_{ij}(1+\alpha)P_{u_i}(u) u_i-P_{u_i u_j}(u) u_i u_j) \Bigr)_{i,j}
$$
For every $\mathbf{z} \in \R^k$, we have that
\begin{eqnarray*}
\mathbf{z}^T \cdot \mathbf{T_u} \cdot \mathbf{z} &=& -\alpha \sum_{i=1}^k \|u_i\|_i^2 z_i^2 +\int_\Omega \mathbf{z}^T \cdot \Bigl( \delta_{ij}(1+\alpha)P_{u_i}(u)u_i-P_{u_i u_j}(u)u_i u_j\Bigr)_{i,j} \cdot \mathbf{z}\, dx\\
										&\leq & -\alpha \sum_{i=1}^k \|u_i\|_i^2 z_i^2
\end{eqnarray*}
by (P4), and hence $\mathbf{T_u}$ is a negative definite matrix. In particular, its determinant is different from zero and the $k$ vectors 
$$F'(u)(u_1,0,\ldots,0),\ldots,F'(u)(0,\ldots,0,u_k)$$
are linearly independent. This implies that $F'(u):\cH\to \R^k$ is onto for every $u\in \cN_*$, and hence $\cN_*$ is indeed a submanifold of $\cH$ of codimension $k$.

As for the second part of the lemma, if $J_\beta|_{\cN_*}'(u)=0$ then there exist real numbers $\lambda_i,\ i=1,\ldots,k$, such that $E'(u)=\sum_{i=1}^k \lambda_i F'_i(u)$. By testing the previous equality with $(0,\ldots,0,u_j,0,\ldots,0)$, one obtains
$$
0=\sum_{i=1}^k \lambda_i \partial_{u_j} F_i(u) u_j,\quad \forall j=1,\ldots,k,
$$
which is equivalent to 
$${\bf T_u} \left( \begin{array}{c} \lambda_1\\ \vdots \\ \lambda_m\end{array}\right)={\bf 0}.$$Hence $\lambda_i=0$ for every $i$, and $u$ is a critical point of the functional $E$.
\end{proof}

\begin{lemma}\label{lemma:Palais-Smale}
$E|_{\cN_*}$ satisfies the Palais-Smale condition.
\end{lemma}

\begin{proof} Here we recover the definitions of $F_i(u)$ and of $\mathbf{T_u}$ from the proof of Lemma \ref{lemma:N_is_manifold}. Let $(u_n)_n=((u_{1,n},\ldots, u_{k,n}))_n\subseteq \cN_*$ be a Palais-Smale sequence for $E |_{\cN_*}$, that is,
$E(u_n)$ remains bounded in $\R$ as $n \to \infty$ and 
\begin{equation}\label{eq:E_N satisfies PS}
E'(u_n)-\sum_{i=1}^k \lambda_{i,n}F'_i(u_n)\rightarrow 0 \text{ in } \cH', \text{ for some sequences } ( \lambda_{1,n} )_n,\ldots, (\lambda_{k,n})_n\subset \R .
\end{equation}
Observe that Lemma \ref{lemma:properties_of_P}-(ii) implies that
$$
\sum_{i=1}^k \|u_{i,n}\|_i^2=\int_\Omega \sum_{i=1}^kP_{u_i}(u_n) u_{i,n}\, dx \geq (2+\alpha)\int_\Omega P(u_n)\, dx,
$$
whence
\begin{equation}\label{eq:lower_bound_for_E}
E(u_n)\geq \Bigl(\frac{1}{2}-\frac{1}{2+\alpha}\Bigr) \sum_{i=1}^k \|u_{i,n}\|_i^2.
\end{equation}
Thus $(u_n)_n$ is bounded in $\cH$ and (up to subsequences), we obtain the existence of $u=(u_1,\ldots, u_k) \in \cH$ such that
\begin{eqnarray*}
u_{i,n} \to u_i \text{ weakly in } H^1_0(\Omega), \text{ strongly in } L^{q}(\Omega) \text{ for every } 2\leq q<2^\ast.
\end{eqnarray*}
By Lemma \ref{lemma:far_away_from_zero}, we deduce that $u_i\not \equiv 0$ for every $i$. Moreover, we have 
$$
\mathbf{T_{u_n}} = \Bigl( \int_\Omega (\delta_{ij}P_{u_i}(u_n)u_{i,n}-P_{u_i u_i}(u_n)u_{i,n}^2)\, dx\Bigr)_{i,j}\to  \Bigl( \int_\Omega (\delta_{ij}P_{u_i}(u)u_{i}-P_{u_i u_i}(u)u_{i}^2)\, dx\Bigr)_{i,j}=:\mathbf{T_u}
$$ 
in $\R^{2k}$ and, for each $i$,
$$
\|u_i\|_i^2\leq \liminf_{n} \|u_{i,n}\|_i^2 =\liminf_n \int_\Omega P_{u_i}(u_{n})u_{i,n}\, dx=\int_\Omega P_{u_i}(u)u_i\, dx.
$$
Hence, by reasoning as in the proof of Lemma \ref{lemma:N_is_manifold}, we obtain that ${\bf T_u}$ is a negative definite matrix, since $\forall \mathbf{z}\in\R^k$,
\begin{eqnarray*}
\mathbf{z}^T \cdot \mathbf{T_u} \cdot \mathbf{z} &=& -\alpha \sum_{i=1}^k  z_i^2 \int_\Omega P_{u_i}(u)u_i\, dx +\int_\Omega \mathbf{z}^T \cdot \Bigl( \delta_{ij}(1+\alpha)P_{u_i}(u)u_i-P_{u_i u_j}(u)u_i u_j\Bigr)_{i,j} \cdot \mathbf{z}\, dx\\
										&\leq & -\alpha \sum_{i=1}^k \|u_i\|_i^2 z_i^2.
\end{eqnarray*}

After testing \eqref{eq:E_N satisfies PS} with $(0, \ldots, u_{j,n},\ldots,0)$ for every $j$,  we obtain, as $n\rightarrow +\infty$,
$$
{\bf o(1)}={\bf T_{u_n}} \cdot\left(\begin{array}{c} \lambda_{1,n}\\ \vdots\\ \lambda_{k,n}\end{array}\right)=\Bigl({\bf T_u}+{\bf  o(1)} \Bigr) \cdot \left(\begin{array}{c} \lambda_{1,n}\\ \vdots\\ \lambda_{k,n}\end{array}\right)
$$
and moreover
\begin{eqnarray*}
\Bigl(\lambda_{1,n},\ldots,\lambda_{k,n}\Bigr) \cdot {\bf o(1)}  &=& \Bigl(\lambda_{1,n},\ldots,\lambda_{k,n}\Bigr)\cdot {\bf T_u} \cdot\left(\begin{array}{c} \lambda_{1,n}\\ \vdots\\ \lambda_{k,n}\end{array}\right) + \Bigl( \lambda_{1,n},\ldots,\lambda_{k,n}\Bigr)\cdot {\bf  o(1)} \cdot \left(\begin{array}{c} \lambda_{1,n}\\ \vdots\\ \lambda_{k,n}\end{array}\right)\\
	&\leq& -C |(\lambda_{1,n},\ldots,\lambda_{k,n})|^2 + \Bigl( \lambda_{1,n},\ldots,\lambda_{k,n}\Bigr)\cdot {\bf  o(1)}\cdot \left(\begin{array}{c} \lambda_{1,n}\\ \vdots\\ \lambda_{k,n}\end{array}\right)
\end{eqnarray*}
for some $C>0$. Thus for every $i$ we have $\lambda_{i,n}\rightarrow 0$ and $\lambda_{i,n}F'_i(u_n)\rightarrow 0$ in $\cH'$ as $n\to \infty$, and therefore also $E'(u_n)\rightarrow 0$ in $\cH'$ as $n\to \infty$. By taking this time $(0,\ldots,u_{i,n}-u_i,\ldots,0)$ as a test function,  we obtain
$$
E'(u_n)(0,\ldots,u_{i,n}-u_i,\ldots,0)={\rm o}(1) \qquad \text{ as } n\to \infty,
$$ 
which is equivalent to 
$$
\langle u_{i,n},u_{i,n}-u_i\rangle_i-\int_\Omega P_{u_i}(u_n)(u_{i,n}-u_i)={\rm o}(1) \qquad \text{ as } n\to \infty.
$$
Since  $\int_\Omega P_{u_i}(u_n)(u_{i,n}-u_i)\, dx \to 0$, it follows that $\|u_{i,n}\|_i\rightarrow \|u_i\|_i$, which provides the strong convergence $u_{i,n}\to u_i$ for every $i$.
\end{proof}

\begin{proof}[Proof of Theorem \ref{thm:existence}] We have that $c
  \geq 0$ (recall \eqref{eq:lower_bound_for_E}). Let $(u_n)_n$ be a
  minimizing sequence for $E|_{\cN_*}$, namely $u_n\in \cN_*$ for all $n$
  and $E(u_n)\rightarrow \inf_{\cN_*}E$ as $n\to\infty$. 
 By the Ekeland's Variational Principle we can suppose, without loss
 of generality, that $(u_n)_n$ is a Palais-Smale sequence for the
 restricted functional $E |_{\cN_*}$. Hence by Lemma
 \ref{lemma:Palais-Smale} we have that, up to a subsequence, $u_n
 \rightarrow u$ strongly in $\cH$. In particular $u\in \cN_*$ (since
 $u_i\not\equiv 0$ for all $i$ by Lemma
 \ref{lemma:far_away_from_zero}). Replacing $u$ with
 $(|u_1|,\ldots,|u_k|)$, we may assume that $u \in \cN$, and by
 (\ref{eq:10}) we have $E(u)=c= \inf_{\cN}E= \inf_{\cN_*}E$. Morover,
 $u$ is a critical point of $E$ by Lemma~\ref{lemma:N_is_manifold}. As
 a consequence of the strong maximum principle, we then see that $u$
 is a solution of (\ref{eq:1}).
\end{proof}


\section{An alternative characterization for the critical level $c$}
\label{sec:an-altern-char}

This section is devoted to the proof of
Theorem~\ref{sec:introduction-1} and related facts. We will assume conditions (P0)--(P4)
from now on. Given $u\in (H^1_0(\Omega)\setminus\{0\})^k$, we consider the function
$$
\vphi=\vphi_u:\R^k\mapsto \R;\quad \vphi(t_1,\ldots, t_k):=E(t_1 u_1,\ldots, t_k u_k)=\sum_{i=1}^k \frac{t_i^2}{2}\|u_i\|_i^2-\int_\Omega P(t_1 u_1,\ldots, t_k u_k)\, dx.
$$
We note that $\vphi$ is even in each variable by (\ref{eq:8}). Moreover, the point $(0,\ldots, 0)$ is always a strict local minimum for the function $\vphi$. In fact,
$$
\vphi(t_1,\ldots, t_k)=\frac{1}{2}\sum_{i=1}^k t_i^2 \|u_i\|^2-\int_\Omega P(t_1u_1,\ldots, t_k u_k)\, dx\geq \frac{1}{2}\sum_{i=1}^k t_i^2 (\|u_i\|_i^2-Ct_i^{p-2}\|u_i\|_i^{p})>0
$$
for sufficiently small $|t_1|+\ldots + |t_k|$. Furthermore we observe
that, if $t_1,\ldots, t_k>0$, then
$$
(t_1u_1,\ldots, t_k u_k)\in \cN_* \qquad  \Leftrightarrow \qquad \nabla \vphi(t_1,\ldots, t_k)=(0,\ldots, 0).
$$

\begin{lemma}\label{lemma:critical_is_maximum}
Let $u\in \cH$ with $u_i\not\equiv 0$ for every $i$ and take $t_1,\ldots, t_k>0$ such that $\nabla \vphi(t_1,\ldots, t_k)=(0,\ldots,0)$. Then $(t_1,\ldots, t_k)$ is a non degenerate local maximum for $\vphi$.
\end{lemma}
\begin{proof}
The proof follows some of the lines of the one of Lemma \ref{lemma:N_is_manifold}. If $(t_1,\ldots, t_k)\in C^+$ is a critical point for $\vphi$, then
$$
\|u_i\|_i^2=\int_\Omega P_{u_i}(t_1u_1,\ldots, t_k u_k)\frac{ t_i u_i}{t_i^2}\, dx \qquad \text{ for every $i$},
$$
and hence the Hessian matrix of $\vphi$ at that point is given by
\begin{eqnarray*}
&&H_\vphi(t_1,\ldots, t_k)=\Bigl(\frac{\partial^2\vphi}{\partial t_i\partial t_j}\Bigr)_{ij}=\Bigl( \delta_{ij}\|u_i\|_i^2-\int_\Omega P_{u_iu_j}(t_1 u_1,\ldots, t_k u_k)u_i u_j\, dx\Bigr)_{ij}\\
				&=&\Bigl(-\alpha \delta_{ij} \|u_i\|_i^2\Bigr)_{ij}\\
				&&+\Bigl(\int_\Omega (\delta_{ij}(1+\alpha)P_{u_i}(t_1u_1,\ldots, t_ku_k)\frac{t_iu_i}{t_i^2}-P_{u_iu_j}(t_1u_1,\ldots, t_k u_k)\frac{t_iu_it_ju_j}{t_it_j} )\, dx\Bigr)_{ij},
\end{eqnarray*}
which is negative definite, since for each $z\in \R^k$ we have
\begin{eqnarray*}
z^T\cdot H_\vphi(t_1,\ldots, t_k)\cdot z &=&-\alpha \sum_{i=1}^k \|u_i\|_i^2 z_i^2 + \int_\Omega \Bigl(\frac{z_1}{t_1},\ldots, \frac{z_k}{t_k}\Bigr)\cdot M(t_1 u_1,\ldots,t_k u_k)\cdot \left(\begin{array}{c} \frac{z_1}{t_1}\\ \vdots\\ \frac{z_k}{t_k}\end{array} \right)\, dx\\
							&\leq& -\alpha \sum_{i=1}^k \|u_i\|_i^2 z_i^2,
\end{eqnarray*}
where we have used (P4) in the last inequality.
\end{proof}

We remark that assumption (P5) was not used in the proof above, but it
will now allow us to control $\vphi$ at the boundary of $C^+$. The geometric
meaning of (P5) can be formulated as follows. Fix $u$ with $u_i\neq 0$ and consider $\psi(t)=E(u_1,\ldots, t u_i, \ldots, u_m)$. Then $\psi'(t)=0$, and $\psi''(0)=\|u_i\|^2-\int_\Omega P_{u_i u_i}(u_1,\ldots, 0,\ldots, u_k)u_i^2\, dx>0$. Hence (P5) implies that
\begin{equation}\label{eq:saddle_point}
E(u_1,\ldots, t,\ldots, u_k)>E(u_1,\ldots, 0,\ldots, u_k) \quad \text{ for sufficiently small t}.
\end{equation}

\begin{prop}\label{prop:unique_maximum}
Let $u\in \cM$. Then the function $\phi=\phi_u$ has precisely one critical
point $(\bar t_1,\ldots, \bar t_k)$ with $\bar t_1,\dots,\bar
t_k>0$. Moreover, $\vphi_u$ attains a global maximum at this point, and
$(\bar t_1 u_1,\ldots, \bar t_k u_k)\in \cN$. 
\end{prop}

\begin{proof}
As $u\in \cM$, we know that $\vphi$ must have a global maximum at a
point $\Lambda_0=(\bar t_1,\ldots, \bar t_k)\in C^+$. As the origin is
a strict local minimum for $\vphi$ and \eqref{eq:saddle_point} holds,
we must have $\bar t_i>0$ for $i=1,\dots,k$, and hence $(\bar t_1 u_1,\ldots,
\bar t_k u_k)\in \cN$. Thus, by the previous lemma, $\Lambda_0$ is a
non degenerate maximum. Suppose now, by contradiction, the
existence of another critical point $\Lambda_1 \in C^+$ having only
positive components. By Lemma~\ref{lemma:critical_is_maximum}, both
$\Lambda_0$ and $\Lambda_1$ are nondegenerate local maxima of $\vphi$.
Hence for 
$$
\bar c=\sup_{A\in \bar \Gamma} \min_{A} \vphi, \qquad \text{ where } \bar \Gamma=\{A\subseteq \R^k:\ A \text{ is compact, connected, and } \Lambda_0, \Lambda_1\in A\},
$$
we have $\bar c<\min \vphi(\Lambda_0),\vphi(\Lambda_1)$. The class $\bar
\Gamma$ was already considered in the paper \cite{AmbRab}. We will now
show the existence of a optimal set in $\bar \Gamma$, which contains a
critical point of $\vphi$ at level $\bar c$. This idea is inspired by the
work \cite{CacSchulz}. Define  
$$
K_{\bar c}=\{(t_1,\ldots, t_k)\in \R^k:\ \vphi(t_1,\ldots, t_k)=\bar c \text{ and } \nabla \vphi(t_1,\ldots, t_k)=(0,\ldots, 0)\}.
$$ 
Since $u \in \cM$, there exists $R>0$ such that 
\begin{equation}
  \label{eq:11}
\vphi(t_1,\dots,t_k)< \bar c-1 \qquad \text{if $|t_1|^2+\dots+|t_k|^2 \ge R^2$.}  
\end{equation}
Let us now put
$$
B:= B_R(0) \cap C^+ \quad \text{and}\quad B_\varepsilon:= \{t \in B_R(0)\::\:
\text{$t_i \ge \varepsilon$ for $i=1,\dots,k$}\} \subset B
$$
for every $\varepsilon>0$. As a consequence of (\ref{eq:saddle_point}) and since $0$ is
a strict local minimum of $\phi$, for $\varepsilon>0$ sufficiently small there exists a map $\psi: B \to
B_\varepsilon$  such that $\phi(\psi(t)) \ge \phi(t)$ for every $t \in B$ and 
$\psi(t)=t$ for every $t \in B_\varepsilon$. We fix $\varepsilon$ and $\psi$ with
this property and such that $\Lambda_0,\Lambda_1 \in B_\varepsilon$. 
We now claim:\\
1. There exists $A_*\in \bar \Gamma$ such that $A_* \subset B_\varepsilon$ and $\min
\limits_{A_*}\vphi=\bar c$.\\
To prove this, take a maximizing sequence for $\bar c$, namely $A_n\in
\bar \Gamma$ such that $\bar c-1/n\leq \min_{A_n}\vphi \leq c$. By
(\ref{eq:11}), we then have $A_n \subset B_R(0)$ for every $n \in
\N$. Therefore the set 
$$
A_*:=\bigcap_{n=1}^\infty \overline{\bigcup_{i=n}^\infty A_i} \subset B_R(0)
$$
is compact and connected, and $\Lambda_0,\Lambda_1\in \bar
A$. Moreover, $\bar c \leq \min_{A_*}\vphi$. Therefore $A_*\in \bar
\Gamma$ and $\min_{A_*}\vphi=\bar c$. As $\vphi$ is even with respect
to each coordinate, we can suppose without loss of generality that
$A_*\subseteq C^+$ and hence $A_* \subset B$. Moreover, replacing
$A_*$ by $\psi(A_*)$ if necessary and recalling that
$\psi(\Lambda_i)=\Lambda_i$ by our choice of $\varepsilon$ and $\psi$, we may
assume that $A_* \subset B_\varepsilon$. 

2. $A_*\cap K_{\bar c}\neq \emptyset$.\\
Suppose this is not true. Then, by the deformation lemma \cite[Theorem 3.4]{struwe}, there exists a neighborhood $\mathcal{V}$ of $K_{\bar c}$ such that $A_*\cap \mathcal{V}= \emptyset$, $\ep<(\vphi(\Lambda_0)-\bar c)/2$, and a homeomorphism $h:\R^k\mapsto \R^k$ such that
\begin{itemize}
\item $h(t_1,\ldots, t_k)=(t_1,\ldots, t_k), \quad |\vphi(t_1,\ldots, t_k)-\bar c|\geq 2\ep$;
\item $\vphi(h(t_1,\ldots, t_k))\geq \bar c+\ep$ for every $(t_1,\ldots, t_k)\notin \mathcal{V}$ such that $\vphi(t_1,\ldots, t_k)\geq \bar c-\ep$.
\end{itemize}
Observe that $h(A_*)$ is a compact and connected set. Moreover, $\vphi(\Lambda_0)=\vphi(\Lambda_1)>\bar c+2\ep$, then $h(\Lambda_0)=u_0, h(\Lambda_1)=u_1$ and $\Lambda_0,\Lambda_1\in h(A_*)$. Hence $h(A_*)\in \bar \Gamma$, and
$$
\bar c+\ep\leq \min_{h(A_*)}\vphi\leq \bar c,
$$
which is a contradiction. Hence $A_*\cap K_{\bar c}\neq \emptyset$, as claimed.\\ 
Now, to reach a final contradiction, let $t=(t_1,\ldots,  t_k) \in A_*\cap
K_{\bar c}$. Since $t_i \ge \varepsilon$ for every $i$, we deduce from
Lemma~\ref{lemma:critical_is_maximum} that $t$ is a strict local
maximum of $\vphi$. Since $A_*$ is connected, this however implies
that 
$\min \limits_{A_*} \vphi < \phi(t)=\bar c$, which contradicts
1. above.
\end{proof}

\begin{proof}[Proof of Theorem~\ref{sec:introduction-1}]
Let $u\in \cM$. By Proposition~\ref{prop:unique_maximum} there exists
$(\bar t_1,\ldots, \bar t_k)$ such that $(\bar t_1 u_1,\ldots, \bar
t_k u_k)\in \cN$ and such that $(\bar t_1,\ldots, \bar t_k)$ is a maximum for $\vphi$ in $C^+$. Hence 
$$
c= \inf_\cN E \le E(\bar t_1 u_1,\ldots, \bar
t_k u_k) \le \sup_{t_1,\dots,t_k \ge 0}E(t_1 u_1,\ldots, t_k u_k)
$$
and this shows (\ref{eq:3}). Moreover, if $u \in \cM$ for some
minimizer $u \in \cN$ of $E|_{\cN}$, then $(1,\ldots,1)$ is a
critical point of $\vphi_u$ and therefore a global maximum of $\vphi$
by Proposition~\ref{prop:unique_maximum}. Hence
$$
\sup_{t_1,\dots,t_k \ge 0}E(t_1 u_1,\ldots, t_k u_k) = E(u)=c,
$$
and therefore equality holds in (\ref{eq:3}).
\end{proof}
 

\section{A general symmetry result for the case of two equations}\label{sec:symmetry results}

Here we will restrict our attention to the two component system
(\ref{eq:system2eq-0}). By the arguments in the beginning of Section
\ref{sec:some-general-notions}, we may assume that $c_1=c_2=1$, so we
are dealing with the system 
\begin{equation}\label{eq:system2eq}
-\Delta u+V_1(x)u=P_u(u,v)\quad -\Delta v+V_2(x)u=P_v(u,v) \qquad u,v\in H^1_0(\Omega).
\end{equation}
We suppose from now on that $\Omega$ is a radial domain, namely a ball
or an annulus, and that $V_1$ and $V_2$ are radial functions,
i.e. $V_i(x)=V_i(y)$ for all $x,y \in \Omega$ with $|x|=|y|$ and
$i=1,2$. As already remarked in the introduction, we cannot expect
ground state solutions of (\ref{eq:system2eq}) to be radial (see
Remark~\ref{sec:some-special-system} below for a
counterexample). However, via polarization methods we will show
Theorem~\ref{coro:main} which states that under the ``negative
coupling assumption'' (P6) ground state solutions are foliated Schwarz
symmetric (as defined in the introduction) in each of their components
with respect to antipodal points. We will state an abstract criterion
for this type of symmetry of solutions of (\ref{eq:system2eq}) first (see
Theorem~\ref{thm:main_result}). This criterion is
of independent interest and has applications within a different
setting, see
Subsection~\ref{sec:further-applications} below.

Let us introduce some useful notations. We define the sets 
$$\cH_0=\{H\subset \R^N:\ H \text{ is a closed half-space in $\R^N$ and $0\in \partial H$}\}$$
and, for $p\neq 0$,
$$\cH_0(p)=\{H\in \cH_0:\ p\in {\rm int}(H)\}.$$
For each $H\in \cH_0$ we denote by $\sigma_H:\R^N\to \R^N$ the reflection in $\R^N$ with respect to the hyperplane $\partial H$, and define the polarization of a function $u:\Omega\to \R$ with respect to $H$ by
$$
u_H(x)=
\left\{
\begin{array}{ll}
\max\{u(x), u(\sigma_H(x))\} & x\in H\cap \Omega,\\
\min\{ u(x),u(\sigma_H(x)) \} & x\in \Omega\setminus H.
\end{array}
\right.
$$
Moreover, we will call $H\in \cH_0$ \emph{dominant} for $u$ if
$u(x)\geq u(\sigma_H(x))$ for all $x\in \Omega\cap H$ (or,
equivalently, $u_H(x)=u(x)$ for every $x\in \Omega\cap H$). On the
other hand we will say that $H\in \cH_0$ is \emph{subordinate} for $u$
if $u(x)\leq u(\sigma_H(x))$ for all $x\in \Omega\cap H$. We recall
from \cite[Lemma 4.2]{brock:03} (see also \cite[Proposition 2.7]{Weth_survey}) the following characterization of foliated Schwarz symmetry.

\begin{prop}\label{prop: equivalent charact for Sch symmetry}
Let $u:\Omega\to \R$ be a continuous function. Then $u$ is foliated Schwarz symmetric with respect to $p\in \partial B_1(0)$ if and only if every $H\in \cH_0(p)$ is dominant for $u$.
\end{prop}

Moreover, we will need the following properties (see for instance \cite[Lemma 3.1]{Weth_survey}).

\begin{lemma}\label{lemma: invariance properties of polarization}
Let $u:\Omega\to \R$ be a measurable function and $H\in \cH_0$.
\begin{itemize}
\item[(i)] If $F: \Omega \times \R\to \R$ is a continuous function
  such that $F(x,t)=F(y,t)$ for every $x,y \in \Omega$ such that
  $|x|=|y|$ and $t \in \R$ and $\displaystyle \int_\Omega |F(x,u(x))|\, dx<+\infty$, then $\displaystyle \int_\Omega F(x,u_H)\, dx=\int_\Omega F(x,u)\, dx$.
\item[(ii)] Moreover, if $u\in H^1_0(\Omega)$ then also $u_H \in
  H^1_0(\Omega)$ and $\displaystyle \int_\Omega |\nabla u_H|^2=\int_\Omega |\nabla u|^2$.
\end{itemize}
\end{lemma}

For every $H\in \cH_0$ we denote by $\widehat H\in \cH_0$ the closure
of the complementary half-space $\R^N\setminus H$. We can now state
the main abstract result of this section.

\begin{thm}\label{thm:main_result}
Take $P\in C^2(\R^2)$ such that
\begin{itemize}
\item[(P6)] $P_{uv}(s,t)<0$ for every $s,t>0$.
\end{itemize}
Let $u,v\in C^2(\Omega)\cap C^1(\overline \Omega)$ be a classical solution of \eqref{eq:system2eq}. If, for every $H\in \cH_0$, the pair $(u_H,v_{\widehat H})$ is also a strong solution of \eqref{eq:system2eq}, then $u$ and $v$ are foliated Schwarz symmetric with respect to antipodal points, that is, there exists $p\in \partial B_1(0)$ such that $u$ is foliated Schwarz symmetric with respect to $p$, and $v$ is foliated Schwarz symmetric with respect to $-p$.
\end{thm}

\begin{proof}
Take $r>0$ such that $\partial B_r(0)\subseteq \Omega$ and let $p\in \partial B_1(0)$ be such that $\max_{\partial B_r(0)}u =u(rp)$. Given $H\in \cH_0(p)$, we will prove that $H$ is dominant for $u$ and subordinate for $v$. This combined with Proposition \ref{prop: equivalent charact for Sch symmetry} immediately provides the conclusion of the theorem. From
$$
-\Delta u+V_1(x)u=P_u(u,v),\quad \text{ and } -\Delta u_H+V_1(x)u_H =P_u(u_H,v_{\widehat H})
$$
it follows that, for $x\in \Omega\cap H$, $w(x):=u_H(x)-u(x)\geq0$ and
\begin{equation}\label{eq: equation for u_H-u}
-\Delta w +c(x)w=P_u(u_H,v_{\widehat H})-P_u(u_H,v),
\end{equation}
with $c(x)=V_1(x)-(P_u(u_H,v)-P_u(u,v))/(u_H-u)\in L_\text{loc}^\infty(\Omega)$. As $v_{\widehat H}\leq v$ in $\Omega \cap H$, condition (P6) implies that $P_u(u_H,v)\leq P_u(u_H,v_{\widehat H})$ in $\Omega\cap H$. Thus 
$$
-\Delta w + c(x)w\geq 0 \qquad \text{ and }\qquad w\geq 0 \quad \text{ in } \Omega\cap H,
$$
which implies (by the Strong Maximum Principle, see for instance
\cite[Theorem 1.7]{Linbook}) that either $w>0$ or $w\equiv 0$ in
$\Omega\cap H$. By the choice of $p$, we have that $rp\in \Omega\cap
H$ and that $w(rp)=0$, and then it must be $u=u_H$ and therefore $w
\equiv 0$ in $\Omega\cap H$. Moreover, coming back to \eqref{eq: equation for u_H-u}, we now see that
\begin{equation}
  \label{eq:18}
P_u(u_H,v_{\widehat H})=P_u(u_H,v)
\end{equation}
and hence, since the map $t\mapsto P_u(s,t)$ is strictly decreasing
for each fixed $s$ as a consequence of (P6), we obtain $v=v_{\widehat H}$ in $\Omega\cap H$. Thus we have proved that $H$ is dominant for $u$ and subordinate for $v$, and the theorem follows.
\end{proof}

\begin{rem}
First we observe that (P6) implies condition (P3). Second, we note
that Theorem \ref{thm:main_result} holds true under slightly more
general assumptions replacing (P6): we can assume instead that 
\begin{quote}
for each $s\geq 0$, the function
$t\mapsto P_u(s,t)$ is nonincreasing in $[0,\infty)$ and strictly
decreasing in $[0,\varepsilon)$ for some $\varepsilon>0$.
\end{quote}
In fact, one can proceed in the previous proof until \eqref{eq:18}. Then, by looking at the second equations of the systems, we would have
$$
-\Delta (v-v_{\widehat H})+  \Bigl(V_2(x)-\frac{P_v(u,v)-P_v(u,v_{\widehat H})}{v-v_{\widehat H}}\Bigr)(v-v_{\widehat H})=0\text{ and } v\geq v_{\widehat H} \text{ in }\Omega\cap H,
$$
which gives that either $v>v_{\widehat H}$ or $v=v_{\widehat H}$ in $\Omega \cap H$. Thus by \eqref{eq:18} and the new assumptions we would have equality.

Alternatively,we could have also supposed that
\begin{quote}
for each $t\geq 0$, the function $s\mapsto P_v(s,t)$ is nonincreasing in $[0,\infty)$ and strictly
decreasing in $[0,\varepsilon)$ for some $\varepsilon>0$.
\end{quote} 
\end{rem}

Before we may complete the proof of Theorem~\ref{coro:main}, we first
need the following lemma.

\begin{lemma}\label{lemma: P(u,v) polarized decreases integral}
Let $P\in C^2(\R^2)$ be such that (P6) holds. Take $u,v>0$ such that $\displaystyle \int_\Omega P(u,v)\, dx<+\infty$. Then for every $H\in \cH_0$ we have that
$$\int_\Omega P(u,v)\, dx \leq \int_\Omega P(u_H, v_{\widehat H})\, dx.$$
\end{lemma}

\begin{proof}
We claim that
$$P(a,c)+P(b,d) \leq P(\max\{a,b\},\min\{c,d\})+P(\min\{a,b\}, \max\{c,d\}) $$ for every $a,b,c,d>0$. In the case that $a\geq b$ and $d\geq c$ the result trivially holds. On the other hand, suppose that $a\geq b$ and $c\geq d$. Then
\begin{multline*}
0 \geq  \int_b^a \int_d^c P_{uv}(\xi, \zeta) \, d\zeta\, d\xi =\int_b^a (P_u(\xi,c)-P_u(\xi,d))\, d\xi\\
 = P(a,c)-P(b,c)-P(a,d)+P(b,d),
\end{multline*}
which proves the claim. From this we conclude that
\begin{eqnarray*}
\int_\Omega P(u,v)\, dx &=& \int_{\Omega\cap H} [ P(u(x),v(x))+P(u(\sigma_H(x)), v(\sigma_H(x))]\, dx\\
&\leq & \int_{\Omega \cap H} [ P(u_H(x), v_{\widehat H}(x))+  P(u_H(\sigma_H(x)), v_{\widehat H}(\sigma_H(x)) ) ]\, dx\\
&=& \int_\Omega P( u_H, v_{\widehat H}  )\, dx.
\end{eqnarray*}
\end{proof}

Finally we may complete the

\begin{proof}[Proof of Theorem~\ref{coro:main}]
Let $(u,v)\in C^2(\Omega,\R^2)\cap C(\overline\Omega,\R^2)$ be a
classical solution of \eqref{eq:system2eq-0} minimizing $E|_\cN$ and
such that $(u,v)\in \cM$. Take $H\in \cH_0$. By
Theorem~\ref{thm:main_result}, we only need to show that $(u_H,v_{\widehat H})$ is also a solution to \eqref{eq:system2eq}. First of all observe that for each $t,s>0$, we have
\begin{eqnarray*}
E(t u_H, s v_{\widehat H})&=&\frac{t^2}{2}\|u_H\|_1^2+\frac{s^2}{2}\|v_{\widehat H}\|_2^2-\int_\Omega P(t u_H,s v_{\widehat H})\, dx\\
								&=& \frac{t^2}{2}\|u_H\|_1^2+\frac{s^2}{2}\|v_{\widehat H}\|_2^2-\int_\Omega P((t u)_H, (s v)_{\widehat H})\, dx\\
								&\leq& \frac{t^2}{2}\|u\|_1^2+\frac{s^2}{2}\|v\|_2^2-\int_\Omega P( t u, s v)\, dx \\
								&=& E( t u, s v),
\end{eqnarray*}
where we have used Lemmas \ref{lemma: invariance properties of polarization} and \ref{lemma: P(u,v) polarized decreases integral}.
Hence, as $(u,v)\in \cM$, we have that also $(u_H,v_{\widehat H})\in \cM$, and so there exists $\bar t, \bar s>0$ such that $(\bar tu_H,\bar sv_{\widehat H})\in \cN$. Therefore, by Proposition \ref{prop:unique_maximum},
\begin{equation*}
c \leq E(\bar t u_H, \bar s v_{\widehat H})\leq E(\bar t u, \bar s v)\leq \max_{t,s\geq 0} E(tu,sv) = E(u,v)=c.
\end{equation*}
and thus $\bar t=\bar s=1$ by the uniqueness of the maximum as stated
in Proposition \ref{prop:unique_maximum}. 
Thus $(u_H, v_{\widehat H})\in \cN$ and $E(u_H, v_{\widehat
  H})=c$. Therefore the second statement in Theorem
\ref{thm:existence} implies that $(u_H, v_{\widehat H})$ is a solution
of \eqref{eq:system2eq}, as required.
\end{proof}


\section{Some special system classes}
\label{sec:applications}

In this section, we will discuss results for special subclasses of system
(\ref{eq:1}), and in particular we will give the proof of Theorem~\ref{coro:main-1} with is
concerned with problem~(\ref{eq:application}). Motivated in particular
by results in the papers \cite{CTV1, CTV2}, we now discuss a general family of
functions $P$ where the interaction terms are seperated from the others. For this let $H\in C^2(\R^k)$ and $f_i\in C^1(\R)$ for $i=1,\ldots, k$. Define $F_i(s):=\int_0^s f_i(\xi)\, d\xi$. For
\begin{equation}
  \label{eq:16}
P \in C^2(\R^k),\qquad P(u)=\sum_{i=1}^k F_i(u_i)-H(u),  
\end{equation}
let us see under which assumptions $P$ satisfies (P1)--(P4). We consider the following assumptions for the functions $f_i$.

\begin{itemize}
\item[(a1)] For each $i$ there exists a constant $C_i>0$ such that
$$
|f_i'(s)|\leq C_i(1+|s|^{p-2})\quad \text{for $s \ge 0$ with some $p
  \in (2,2^\ast)$,}
$$
where $2^\ast=2N/(N-2)$ if $N\geq 3$, $2^\ast=+\infty$ otherwise.
\item[(a2)] $f_i(s)={\rm o}(s)$ as $s\to 0$, for every $i=1,\ldots, k$.
\item[(a3)] There exists $\gamma>0$ ($2+\gamma\leq p$) such that
$$
0<(1+\gamma)f_i(s)s\leq f_i'(s)s^2, \qquad \text{for all $s \ge 0$}.
$$
\end{itemize}
Moreover, for the interaction potential $H$ we assume the following.
\begin{itemize}
\item[(H1)] There exist constants $C>0$ and $0<\alpha\leq \gamma$ such that
$$
|H_{u_i u_j}(u)|\leq  C(1+\sum_{i=1}^k |u_i|^\alpha),\qquad \text{for $i,j\in \{1,\ldots, k\},\ u\in C^+.$}
$$
\item[(H2)] $H(0)=0$ and $H_{u_i}(u_1,\ldots,
  u_{i-1},0,u_{i+1},\ldots, u_k)=0$ for $i=1,\dots,k$ and $u \in C^+$.
\item[(H3)] $H_{u_i}(u) \geq 0$ for $i=1,\dots,k$ and $u\in C^+$.
\item[(H4)] For every $u \in C^+$, the matrix 
$$
(h_{ij})_{ij}=\Bigl( \delta_{ij} (1+\alpha) H_{u_i}(u)u_i-H_{u_i u_j}(u) u_i u_j\Bigr)_{i,j=1,\ldots, k}
$$
is positive semidefinite, where $\alpha$ is the constant appearing on (H1). \footnote{Actually this is equivalent to ask (H1) and (H4) for two different constants $\alpha_1,\alpha_2\leq \gamma$, as in each case if each assumption is true for some $\beta$, it is true for every $\bar \beta\geq \beta$.}
\end{itemize}

We then have the following result.

\begin{thm}
\label{sec:applications-2}
Let $f_i$ satisfy (a1)--(a3) and $H$ satisfy ($H1$)--($H4$). Then
(P1)--(P5) hold for $P$ defined in~(\ref{eq:16}). Hence, if
the functions $V_i \in L^\infty(\Omega)$, $i=1,\dots,k$, satisfy $(P0)$, then the assertions of
Theorems~\ref{thm:existence} and \ref{sec:introduction-1} are true. In particular, the system
$$\left\{\begin{array}{l}
-\Delta u_i+V_i(x)u=f_i(u_i)-H_{u_i}(u)\\
u_i\in H^1_0(\Omega)(u),\ u_i>0\ \text{ in }\Omega.
\end{array}\right. \qquad i=1,\ldots, k.
$$
admits a non-trivial solution which minimizes the functional
$E|_{\cN}$.\\
Moreover, if in addition $\alpha< \gamma$ in (H1), then every $u \in
\cH$ with $u_i \ge 0, u_i \not \equiv 0$ for
$i=1,\dots,k$ is contained in $\cM$, and therefore equality holds in
(\ref{eq:3}).
\end{thm}

\begin{proof}
(P1) is an immediate consequence of $(a1)$ and $(H1)$, and
(P2) is an immediate consequence of $(a2)$ and $(H2)$.
(P3) follows directly from (H3), and (P4) follows directly from (a3) and (H4). As for (P5), observe that
$$
P_{u_iu_i}(u_1,\dots,u_{i-1},0,u_{i+1},\dots,u_k)=f_i'(0)-\lim_{t\to 0^+} \frac{H_{u_i}(u_1,\dots,u_{i-1},t,u_{i+1},\dots,u_k)}{t}\leq f_i'(0)=0,
$$ 
for $i=1,\dots,k$ by (a2) and (H3). As a consequence, the assertions of
Theorems~\ref{thm:existence} and \ref{sec:introduction-1} are true.\\
Finally, let us assume that $\alpha< \gamma$ holds in assumption
(H1), and let $u \in \cH$ with $u_i \ge 0, u_i \not \equiv 0$ for
$i=1,\dots,k$. We show that $u \in \cM$. For this we note
that condition (a3) implies the existence of constants $C_i,D_i>0$ such that
$$
F_i(t)\geq C_i t^{2+\gamma} - D_i, \quad \forall s\geq 0,\:i=1,\dots,k.
$$
Thus, for some constant $C_1>0$,
\begin{eqnarray*}
E(t_1 u_1,\dots,t_k u_k) &=& \sum_{i=1}^k
\Bigl(\frac{t_i^2}{2}\|u_i\|_i^2-\int_\Omega F_i(t_iu_i)\, dx\Bigr) +
\int_\Omega H(t_1 u_1,\dots,t_k u_k)\, dx\\
		&\leq& \sum_{i=1}^k
\Bigl(\frac{t_i^2}{2}\|u_i\|_i^2 - C_i t_i^{2+\gamma} \int_\Omega
|u_i|^{2+\gamma}\,dx + Ct_i^{2+\alpha} \int_\Omega |u_i|^{2+\alpha}\,
dx \Bigr) +C_1 \to -\infty,
\end{eqnarray*}
as $|t_1|+\dots+|t_k|\to +\infty$, since $\gamma>\alpha$ and $\int_\Omega
|u_i|^{2+\gamma}\,dx>0$ for $i=1,\dots,k$.
\end{proof}

Theorem~\ref{sec:applications-2} generalizes the existence result
contained in \cite[Theorem 2.1]{CTV1} and \cite[Theorem
2.2]{CTV2}. The main difference is that we allow $\alpha=\gamma$ in (H1), which means that we allow $F_i$ and $H$ to have the
same kind of growth at infinity. We point out that in this case it is
not necessarily true that $u \in \cM$ for every $u \in 
\cH$ with $u_i \ge 0, u_i \not \equiv 0$ for
$i=1,\dots,k$. As an example, consider the two-component system
\begin{equation}
  \label{eq:17}
-\Delta u_1 = u_1^3 -\beta u_2^2 u_2,\quad -\Delta u_2 = u_2^3 -\beta u_1^2 u_2
\qquad u_1,u_2 \in H^1_0(\Omega)
\end{equation}
which in dimension $N \le 3$ and for $\beta>0$ is a special case of assumptions
(a1)--(a3), (H1)--(H4) with $f_1(t)=f_2(t)=t^3$,
$H(u_1,u_2)=\frac{\beta}{2} u_1^2 u_2^2$ and $\alpha=\gamma=2$. The
corresponding energy functional is then given by 
$$
u=(u_1,u_2) \mapsto E(u)=\sum_{i=1}^2 \int_\Omega \Bigl(\frac{1}{2}|\nabla
u_i|^2 -\frac{1}{4} |u_i|^4\Bigr)\,dx + \frac{\beta}{2} \int_\Omega
|u_1|^2|u_2|^2\,dx,
$$ 
and in case $\beta \ge 1$ we have $E(tw,tw) \to +\infty$ as $t \to
\infty$ for every $w \in H^1_0(\Omega) \setminus \{0\}$, so that
$(w,w) \not \in \cM$. Nevertheless, we will be able to show $\cN \subset \cM$ for system
(\ref{eq:17}) and the more general class of systems
(\ref{eq:application}) arising from the choice of functions 
\begin{equation}
  \label{eq:15}
f_i(u)= \lambda_i u^{p-1} \quad \text{and}\quad H(u)= \frac{1}{2}
\sum_{\stackrel{i,j=1}{i \not=j}}^k \beta_{ij}u_i^{q_i}u_j^{q_j},
\end{equation}
where $2<p<2^*$ and the other parameters satisfy (\ref{eq:4}). This
also leads to
equality in (\ref{eq:3}) and therefore to a minimax characterization
of $\inf_\cN E$.

\begin{prop}
\label{sec:applications-1}  
The class of functions given by (\ref{eq:15}) satisfies assumptions
$(a1)$--$(a3)$ and $(H1)$--$(H4)$ with $\alpha=\gamma=p-2$. Moreover, if
the functions $V_i \in L^\infty(\Omega)$, $i=1,\dots,k$, satisfy $(P0)$,
then we have $\cN \subset \cM$ in this case, where $\cN$ and $\cM$ are
defined with respect to the corresponding functional
$$
u \mapsto E(u)=\frac{1}{2}\sum_{i=1}^k \|u_i\|_i^2 -
\sum_{i=1}^k \frac{\lambda_i}{p} \int_\Omega |u_i|^p\,dx + \sum_{i \not= j}\beta_{ij}
                \int_\Omega  |u|^{q_i}|v|^{q_j}\, dx.
$$ 
\end{prop}

\begin{proof}
Assumptions $(a_1)$--$(a_3)$ and $(H1)$--$(H3)$ are rather immediate. We now show that also
(H4) holds with the choice $\alpha=p-2$. Let $u
\in C^+$ and recall the matrix $(h_{ij})_{ij}$ defined in (H4). We have for each $i$
$$
h_{ii}=(1+\alpha)H_{u_i}(u)u_i-H_{u_i u_i}(u)u_i^2=(p-q_i) q_i
u_i^{q_i}\sum_{j\neq i}\beta_{ij}u_j^{q_j}>0
$$
and, for $j\neq i$,
$$
h_{ij}=- H_{u_i u_j}(u)u_iu_j=-q_iq_j \beta_{ij} u_i^{q_i}u_j^{q_j}<0.
$$
By the Gershgorin's theorem (see for instance \cite[Appendix 7]{lax}), the eigenvalues of $(h_{ij})_{ij}$ lie in the set
\begin{eqnarray*} 
\bigcup_{i=1}^k \Bigl\{\lambda:\ |\lambda-h_{ii}| \leq  \sum_{j\not=i}
|h_{ij}|\Bigr\} &\subseteq& \bigcup_{i=1}^k \Bigl\{\lambda: \lambda \geq
h_{ii}+ \sum_{j=1}^k h_{ij} \Bigr\}\\
&=&\bigcup_{i=1}^k  \Bigl\{\lambda: \lambda \geq   \sum_{j\neq i}
\beta_{ij} q_i(p-q_i-q_j)u_i^{q_i}u_j^{q_j}\Bigr\}.
\end{eqnarray*}
Hence (\ref{eq:4}) implies that all eigenvalues of $(h_{ij})_{ij}$ are
nonpositive, and hence $(h_{ij})_{ij}$ is a negative semidefinite matrix.\\
To show that $\cN \subset \cM$, let $u \in \cN$. For $(t_1,\dots,t_k)
\in C^+$, we then have, by Young's inequality, 
$$
t_i^{q_i} t_j^{q_j}= \frac{q_i}{p}t_i^{p}+\frac{q_j}{p}t_j^{p}+
\kappa_{ij} \qquad \text{for $i,j=1,\dots,k$ with $\kappa_{ij}= 1-\frac{q_i+q_j}{p} \ge 0$.}
$$
Consequently,
\begin{align*}
E(&t_1u_1,\dots,t_k u_k)= \frac{1}{2}\sum_{i=1}^k t_i^2 \|u_i\|_i^2 -
\sum_{i=1}^k \frac{\lambda_i}{p} t^p \int_\Omega |u_i|^p\,dx + \sum_{i \not= j}t_i^{q_i}t_j^{q_j}\beta_{ij}
                \int_\Omega  |u|^{q_i}|v|^{q_j}\, dx\\
&\leq 
\frac{1}{2}\sum_{i=1}^k t_i^2 \|u_i\|_i^2 -
\sum_{i=1}^k \frac{\lambda_i}{p} t_i^p \int_\Omega |u_i|^p\, dx+ \sum_{i \not= j}\Bigl(\frac{q_i t_i^{p}+q_j t_j^{p}}{p}+\kappa_{ij}\Bigr)\beta_{ij}
                \int_\Omega  |u|^{q_i}|v|^{q_j}\, dx\\
&= 
\frac{1}{2}\sum_{i=1}^k t_i^2 \|u_i\|_i^2 - \sum_{i=1}^k \frac{t_i^p}{p}
\Bigl(\lambda_i \int_\Omega |u_i|^p\, dx - q_i \sum_{j \not= i}
\beta_{ij} \int_\Omega  |u|^{q_i}|v|^{q_j}\, dx\Bigr)+\kappa\\
&= 
\sum_{i=1}^k \Bigl(\frac{t_i^2}{2}-\frac{t_i^p}{p}\Bigr) \|u_i\|_i^2 +\kappa \to -\infty
\qquad \text{as $t_1+\dots +t_k \to +\infty$}
\end{align*}
with $\kappa= \sum \limits_{i \not= j}\kappa_{ij}\beta_{ij}
\int_\Omega  |u|^{q_i}|v|^{q_j}\, dx$, where in the last step we have
used that $u \in \cN$. This shows $u \in \cM$, and we conclude that $\cN \subset \cM$.
\end{proof}

We may now complete the 

\begin{proof}[Proof of Theorem~\ref{coro:main-1}]
By Proposition~\ref{sec:applications-1} and
Theorem~\ref{sec:applications-2}, assumptions (P1)-(P5) are satisfied
for $P$ given in (\ref{eq:5}). Hence Theorem~\ref{thm:existence}
implies that $\inf_\cN E$ is attained, and that every minimizer $u \in \cN$ of $E|_{\cN}$ 
is a weak solution of (\ref{eq:application}). Moreover, by elliptic
regularity, noting that the right hand side of (\ref{eq:system2eq}) is
H{\"o}lder continuous, we find that $u 
\in C^2(\Omega,\R^k)\cap C(\overline\Omega,\R^k)$ is in fact a
classical solution. Since we also know from
Proposition~\ref{sec:applications-1} that $\cN \subset \cM$,
Theorem~\ref{sec:introduction-1} implies that  
$$ 
\inf_\cN E= \inf_{u\in \cM} \sup_{t_1,\ldots, t_k \ge 0} E(t_1
u_1,\ldots, t_k u_k),  
$$
and in case $k=2$ with $\Omega$, $V_1,V_2$ radially symmetric, it
follows from Theorem \ref{coro:main} that every $u 
\in C^2(\Omega,\R^k)\cap C(\overline\Omega,\R^k)$ minimizing $E$ on
$\cN$ is such that $u$ and $v$ are foliated Schwarz symmetric with respect to antipodal points.
\end{proof}

We add a symmetry result corresponding to the class of functions (\ref{eq:16})
in the case $k=2$ under the extra assumption that $\alpha<\gamma$ in (H1).
Hence we consider a system of the type
\begin{equation}\label{eq:H_uH_v_2eq}
\left\{\begin{array}{l}
-\Delta u=f_1(u)-H_u(u,v),\\
 -\Delta v=f_2(v)-H_v(u,v)\\
u,v\in H^1_0(\Omega),\ u,v>0 \text{ in }\Omega.
\end{array}\right.
\end{equation}

\begin{thm}\label{thm:symmetry_with_H}
Take $f_1,f_2$ satisfying (a1)-(a3) and $H$ satisfying ($\tilde H1$),
(H2)--(H4) and 
\begin{itemize}
\item[(H5)] $H_{uv}(s,t)>0$ for every $s,t>0$.
\end{itemize}
Furthermore, suppose that $\Omega$ is radially symmetric, and that
$V_1, V_2 \in L^\infty(\Omega)$ are radial functions satisfying (P0).  
Let $(u,v) \in \cN$ be a minimizer of $E|_\cN$. Then $u$ and $v$ are foliated Schwarz symmetric with respect to antipodal points.
\end{thm}

\begin{proof}
This is a direct consequence of Theorem~\ref{coro:main}, since the
second statement of 
Theorem~\ref{sec:applications-2} implies that $(u,v) \in \cM$ as a
consequence of assumption ($\tilde H1$). 
\end{proof}

\begin{rem}
\label{sec:some-special-system}
In general, minimal energy solution to \eqref{eq:H_uH_v_2eq} are not radial. So see this, let us rewrite the system \eqref{eq:H_uH_v_2eq} with an extra parameter $\beta>0$
\begin{equation}\label{eq:last_system}
\left\{\begin{array}{l}
-\Delta u=f(u)-\beta H_u(u,v),\\
 -\Delta v=f(v)-\beta H_v(u,v)\\
u,v\in H^1_0(\Omega),\ u,v>0 \text{ in }\Omega.
\end{array}\right.
\end{equation}
Suppose that $\Omega$ is either a ball or an annulus. Again, suppose that $f$ satisfy (a1)--(a3), and $H$ satisfy (H1)-(H4). For each $\beta>0$, denote by $E_\beta$ and $\cN_\beta$ the associate energy functional and Nehari manifold. Take $(u_\beta, v_\beta)$ to be a family of positive solutions of \eqref{eq:last_system} minimizing $E_\beta|_{\cN_\beta}$. Then, by the results shown in \cite{CTV1, CTV2}, we know that there exists $\bar u, \bar v>0$ such that $u_\beta\to \bar u, v_\beta\to v$ strongly in $H^1_0(\Omega)$, and $\bar w:=\bar u-\bar v$ satisfies
$$
-\Delta \bar w=f(\bar w) \qquad  \quad J(\bar w)=\min\{J(w):\ w^{\pm}\not\equiv 0,\ J'( w) w^+=J'( w) w^-=0\},
$$
with $J(w)=\frac{1}{2}\int_\Omega |\nabla w|^2\, dx-\int_\Omega F(w)\, dx$. Thus $\bar w$ is a \emph{least energy nodal solution} which, by \cite[Theorem 1.3]{AftalionPacella}, is know to be non radial. Therefore we conclude, from the strong convergence, that $(u_\beta,v_\beta)$ are non radial solutions, at least for sufficiently large $\beta$. 
\end{rem}

\subsection{An application within in a different variational setting}
\label{sec:further-applications}
We close this paper with an application of Theorem
\ref{thm:main_result} which does not fit in the framework of
Theorem~\ref{coro:main}. Consider the cubic system
\begin{equation}\label{eq: CLLL}
\left\{
\begin{array}{ll}
-\Delta u  =\lambda u - u^3-\beta u v^2 & {\rm in }\ \Omega\\[5pt]
-\Delta v = \mu v- v^3- \beta u^2 v & {\rm in }\ \Omega\\[5pt]
u,v\in H^1_0(\Omega),\ u,v>0\ \text{ in }\Omega,
\end{array}
\right.
\end{equation}
where we consider $\beta>0$.
Observe that due to the sign of the pure nonlinearities, this is not a particular case of \eqref{eq:application}.
Following \cite{CLLL}, in this case a minimal energy solutions is defined as a minimizer of the functional
$$I(u,v)=\frac{1}{2}\int_\Omega (|\nabla u|^2+|\nabla v|^2)\, dx + \frac{1}{4}\int_\Omega (u^4+v^4)\, dx + \frac{\beta}{2} \int_\Omega u^2 v^2\, dx $$
constrained to the manifold
$$\cS=\{ (u,v)\in H^1_0(\Omega)\times H^1_0(\Omega):\ \int_\Omega u^2\, dx =\int_\Omega v^2\, dx=1 \}$$
(which represents a mass conservation law). With this framework, $\lambda$ and $\mu$ are understood as Lagrange multipliers, and
\begin{equation}\label{eq: lambda and mu}
\lambda=\lambda(u,v)=\int_\Omega (|\nabla u|^2+u^4 + \beta u^2 v^2)\, dx,\qquad  \mu=\mu(u,v)=\int_\Omega (|\nabla v|^2+v^4 + \beta u^2 v^2)\, dx.
\end{equation}
By using direct methods and the maximum principle, it is easy to prove that \eqref{eq: CLLL} admits a positive solution, minimizer of $I|_{\cS}$.

\begin{thm}
Let $u,v>0$ be minimizers of $I|_{\cS}$, hence in particular solutions of \eqref{eq: CLLL}. Then $u$ and $v$ are foliated Schwarz symmetric with respect to antipodal points.
\end{thm}
\begin{proof}
We start with the observation that $(u,v)$ solve \eqref{eq: CLLL} with $\lambda,\mu$ given by \eqref{eq: lambda and mu}. For every $H\in \cH_0$, by Lemma \ref{lemma: invariance properties of polarization}-(i) we deduce that $(u_H,v_{\widehat H})\in \cS$. Moreover, Lemma \ref{lemma: P(u,v) polarized decreases integral} applied to the map $(u,v)\mapsto u^2v^2$ gives 
\begin{eqnarray}\label{eq:1-1}
\min_{\cS} I&\leq & I(u_H,v_{\widehat H})=\frac{1}{2}\int_\Omega (|\nabla u_H|^2+|\nabla v_{\widehat H}|^2)\, dx + \frac{1}{4}\int_\Omega (u_H^4+v_{\widehat H}^4)\, dx +\frac{\beta}{2}\int_\Omega u_H^2 v_{\widehat H}^2\, dx \nonumber \\
&=& \frac{1}{2}\int_\Omega (|\nabla u|^2+|\nabla v|^2)\, dx + \frac{1}{4}\int_\Omega (u^4+v^4)\, dx +\frac{\beta}{2}\int_\Omega u_H^2 v_{\widehat H}^2\, dx \nonumber \\
&\leq & \frac{1}{2}\int_\Omega (|\nabla u|^2+|\nabla v|^2)\, dx + \frac{1}{4}\int_\Omega (u^4+v^4)\, dx +\frac{\beta}{2}\int_\Omega u^2 v^2\, dx \nonumber \\
&= & I(u,v)=\min_{\cS} I.
\end{eqnarray}
Thus $(u_H,v_{\widehat H})\in \cS$ and $I(u_H,v_{\widehat H})=\min_{\cS}I$, and in particular $(u_H,v_{\widehat H})$ solves \eqref{eq: CLLL} with 
$$
\lambda=\lambda(u_H,v_{\widehat H})=\int_\Omega (|\nabla u_H|^2+u_H^4 + \beta u_H^2 v_{\widehat H}^2)\, dx,\qquad  \mu=\mu(u_H,v_{\widehat H})=\int_\Omega (|\nabla v_{\widehat H}|^2+v_{\widehat H}^4 + \beta u_H^2 v_{\widehat H}^2)\, dx.
$$
Again from \eqref{eq:1-1} we deduce that actually
$$\int_\Omega u_H^2 v_{\widehat H}^2 =\int_\Omega u^2 v^2\, dx$$
and hence $\lambda(u,v)=\lambda(u_H,v_{\widehat H})$, $\mu(u,v)=\mu(u_H,v_{\widehat H})$. Thus $(u,v)$ and $(u_H,v_{\widehat H})$ solve the same system and Theorem \ref{thm:main_result} applies.
\end{proof}

\noindent {\bf Acknowledgments.}
H. Tavares was supported by Funda\c c\~ao para a Ci\^encia e a Tecnologia, grant SFRH/BPD/69314/2010  and PEst OE/MAT/UI0209/2011.\\
This paper started during a visit of H. Tavares to the University of Frankfurt. His stay was also partially supported by a grant of the
Justus-Liebig-University, Giessen\\
Moreover, the authors would like to thank S. Terracini for useful discussions around this subject.

\noindent \verb"htavares@ptmat.fc.ul.pt"\\
University of Lisbon, CMAF, Faculty of Science, Av. Prof. Gama Pinto
2, 1649-003 Lisboa, Portugal

\noindent \verb"weth@math.uni-frankfurt.de"\\
Institut f\"ur Mathematik, Goethe-Universit\"at Frankfurt, Robert-Mayer-Str. 10,
D-60054 Frankfurt a.M., Germany

\end{document}